\documentclass[11pt]{amsart}
\usepackage[colorlinks,bookmarksopen=true]{hyperref}
\usepackage[hmargin=32mm, vmargin=27mm, includefoot, twoside]{geometry}
\usepackage[english]{babel}
\usepackage{amssymb}
\usepackage{mathrsfs}
\theoremstyle{plain}
\newtheorem{prop}{Proposition}[section]
\newtheorem{thm}[prop]{Theorem}
\newtheorem*{thm*}{Theorem}

\newtheorem*{addendum*}{Addendum}
\newtheorem{cor}[prop]{Corollary}

\newtheorem*{convention*}{Convention}
\theoremstyle{definition}
\newtheorem*{defn*}{Definition}

\newtheorem{remark}[prop]{Remark}

\newtheorem*{scholium*}{Scholium}
\theoremstyle{remark}
\newtheorem{example}[prop]{Example}
\newtheorem*{example*}{Example}
\numberwithin{equation}{section}

\newcommand{\F}{\mathscr{F}}
\newcommand{\orient}{\mathscr{O}}
\newcommand{\sD}{\mathscr{D}}
\newcommand{\CC}{\mathbf{C}}
\newcommand{\NN}{\mathbf{N}}

\newcommand{\RR}{\mathbf{R}}
\newcommand{\ZZ}{\mathbf{Z}}
\newcommand{\goth}{\mathfrak{g}}
\newcommand{\gothh}{\mathfrak{h}}
\newcommand{\inv}{^{-1}}
\newcommand{\rank}{\mathrm{rk}}
\newcommand{\lstar}{\ell^2_\bigstar}
\newcommand{\lcirc}{\ell^2_\bigcirc}
\newcommand{\lalt}{\ell^2_{\mathrm{alt}}}
\newcommand{\se}{\subseteq}

\newcommand{\graph}{\mathscr{G}}
\newcommand{\width}{\mathrm{width}}
\newcommand{\cost}{\mathscr{C}}
\newcommand{\betti}{\beta^1} %        see if we want to get a subscript (2) or not...

\newcommand{\emf}[1]{{\bfseries\emph{#1}}}

\begin{document}
\title[Non-unitarisable representations and random forests]{Non-unitarisable representations\\ and random forests}
\author[I. Epstein]{Inessa Epstein*}
\address{I.E.: EPFL, 1015 Lausanne, Switzerland}
\curraddr{Caltech, Pasadena, CA 91125, USA}
%\email{iepstein@math.ucla.edu}
\thanks{*Supported in part by the NSF and the Clay Institute}
\author[N. Monod]{Nicolas Monod$^\dagger$}
\address{N.M.: EPFL, 1015 Lausanne, Switzerland}
%\email{nicolas.monod@epfl.ch}
\thanks{$^\ddagger$Supported in part by the Swiss National Science Foundation}
%\date{19 May, 2008}
%\subjclass{}
%\keywords{}
%
\begin{abstract}
We establish a connection between Dixmier's unitarisability problem and the expected degree of random forests on a group. As a consequence, a residually finite group is non-unitarisable if its first $L^2$-Betti number is non-zero or if it is finitely generated with non-trivial cost. Our criterion also applies to torsion groups constructed by D.~Osin, thus providing the first examples of non-unitarisable groups not containing a non-Abelian free subgroup.
\end{abstract}
\maketitle
\let\languagename\relax  % TO FIX A BUG IN RUNNING HEADERS AND BABEL

%===================================================================================================
%===================================================================================================
\section{Introduction}
If an operator $P$ is conjugated to a unitary operator, then it is uniformly bounded in the sense that $\sup_{n\in \ZZ} \|P^n\|$ is finite. The classical 1947 article by B.~Sz\H{o}kefalvi-Nagy~\cite{Sz-Nagy47} establishes the converse. A remarkable feature of Sz.-Nagy's short proof is that it uses the Banach--Mazur ``generalised limits''.

\smallskip

More generally, a representation $\pi$ of a group $G$ on a Hilbert space $V$ is called \emf{unitarisable} if there is an invertible operator $T$ of $V$ such that $T\pi(g) T\inv$ is unitary for all $g\in G$. In that case, $\pi$ is necessarily \emf{uniformly bounded} in the sense that $\sup_{g\in G} \|\pi(g)\|$ is finite%, indeed bounded by $\|T\|\cdot\|T\inv\|$
. Both J.~Dixmier~\cite{Dixmier50} and M.~Day~\cite{Day50} noticed that the very proof of Sz.-Nagy establishes that every amenable group is \emf{unitarisable}, meaning that all its uniformly bounded representations are unitarisable. Indeed, a group is \emf{amenable} by definition if it admits an invariant mean, \emph{i.e.} a generalised Banach limit.%) on the space $\ell^\infty(G)$ of bounded functions. 

\medskip

J.~Dixmier asked in~\cite[\S5]{Dixmier50} whether unitarisability characterises amenability; the present note contributes to this question. For more background, we refer to~\cite{PisierLNM}.

\medskip

A property very much opposed to amenability is the non-vanishing of the first \emf{$L^2$-Betti number} $\betti$ (to be briefly recalled below; for a detailed discussion, see~\cite{Eckmann00, Lueck}).

\begin{thm}\label{thm:betti}
Let $G$ be a residually finite group. If~$\betti(G) > 0$, then $G$ is not unitarisable.
\end{thm}

%\begin{remark}\label{rem:finite_generation}
%Residual finiteness is probably not an optimal assumption; if $G$ is finitely generated, we can weaken it to admitting infinitely many finite quotients.
%\end{remark}

A similar property is that the \emf{cost} $\cost$ studied in~\cite{Gaboriau00} be larger than one. As suggested by M.~Ab\'ert, the arguments leading to the previous result have a parallel with cost.

\begin{thm}\label{thm:cost}
Let $G$ be a finitely generated residually finite group. If~$\cost(G) > 1$, then $G$ is not unitarisable.
\end{thm}

In fact, Dixmier first asked whether any group at all fails to be unitarisable; this was answered in 1955 when Ehrenpreis--Mautner~\cite{Ehrenpreis-Mautner} showed that the complementary series of $\mathrm{SL}_2(\RR)$ can be extended to uniformly bounded representations that are not unitarisable. A detailed treatment was given by Kunze--Stein~\cite{Kunze-Stein}.

By general properties of unitarisability, the existence of any non-unitarisable group implies that the free group $F_2$ is non-unitarisable (see~\cite{PisierLNM}). Very explicit non-unitarisable representations of $F_2$ were constructed in the eighties~\cite{Mantero-Zappa83,Pytlik-Swarc,Bozejko87}. It follows by induction of representations that any group containing $F_2$ as a subgroup is non-unitarisable.

\medskip

Until now, there was no example of non-unitarisable group not containing $F_2$. In fact, even the existence non-amenable groups without $F_2$ subgroup was a long-standing open problem in group theory, not solved until the 1980's~\cite{Olshanskii80,Adyan83}.

\bigskip

We aim to construct non-unitarisable representations under weaker assumptions than the existence of a free subgroup. A result of Gaboriau--Lyons~\cite{Gaboriau-Lyons}, notably using~\cite{Hjorth_attained} and~\cite{Pak-Smirnova-Nagnibeda}, provides an $F_2$-action on the Bernoulli percolation of any non-amenable countable group $G$ in such a way that $F_2$ can be thought of as a ``random subgroup'' of $G$, even when $G$ has no actual such subgroup. It was suggested in~\cite{MonodICM} (Problem~N) to apply an induction procedure for specific representations of random subgroups in order to answer Dixmier's question. In fact, a first use of~\cite{Gaboriau-Lyons} towards a cohomological question asked in~\cite[\S10]{Johnson} can be found in~\cite[\S5]{MonodICM} and a second use is the ergodic-theoretical result~\cite{Epstein07}.

\smallskip

We shall follow the above strategy, using the language of random forests. A \emf{forest} on a group $G$ is a subset $F\se G\times G$ such that the resulting graph $(G, F)$ has no cycles. The collection $\F_G$ of all forests on $G$ is a closed $G$-invariant subspace of the compact $G$-space of all subsets of $G\times G$ if we consider the usual product topology (\emph{i.e.} pointwise convergence) and the left diagonal $G$-action. A \emf{random forest} is a $G$-invariant Borel probability measure on $\F_G$. By $G$-invariance, the expected degree of a vertex in a random forest does not depend on the vertex; we call it the \emf{expected degree $\deg(\mu)$ of the random forest $\mu$}. Similarly, we define the \emf{width} $\width(\mu)$ as the number of vertices that neighbour a given vertex with positive probability. We shall be interested in forests with finite width. Of course, one has $\deg(\mu)\leq \width(\mu)$.

\begin{thm}\label{thm:forests}
Let $G$ be a unitarisable group. Then the quantity
$$\frac{\deg(\mu)^2}{\width(\mu)}$$
is bounded uniformly over all random forests $\mu$ (of finite width) defined on all countable subgroups of $G$.
\end{thm}

\begin{remark}
We only made the countability assumption in order to have a metrisable space of forests on which the probability is defined. This is an inessential restriction; in any case, unitarisability is a countably determined property~\cite[0.10]{Pisier_survey}. Notice also that all trees in a forest of finite width are countable.
\end{remark}

Using known estimates on specific random forests, Theorem~\ref{thm:forests} implies the following statement, wherein the \emf{rank} $\rank(H)$ denotes the minimal number of generators of a group $H$.

\begin{thm}\label{thm:betti_bis}
Let $G$ be a unitarisable group. Then the quantities
$$\frac{\big(\betti(H)\big)^2}{\rank(H)}, \kern10mm \frac{\big(\cost(H)\big)^2}{\rank(H)}$$
are bounded uniformly over all finitely generated subgroups $H$ of $G$.
\end{thm}

We asked D.~Osin whether one knows examples of groups without non-Abelian free subgroup and violating the above bound involving $\beta^1$. It turns out that D.~Osin can construct \emph{torsion} groups with this property (using among others~\cite{Peterson-Thom}); for this and more, we refer to the forthcoming~\cite{Osin09}. Thus, Osin's examples allow to deduce the following from Theorem~\ref{thm:betti_bis}.

\begin{cor}\label{cor:torsion}
There exist non-unitarisable torsion groups.
\end{cor}

We shall begin by proving Theorem~\ref{thm:forests} in Section~\ref{sec:forests}. This result makes it desirable to investigate general constructions of forests with large expected degree. Indeed, Theorems~\ref{thm:betti} and~\ref{thm:cost} will be deduced by considering specific models of random forests and using known estimates for their degrees. In Section~\ref{sec:betti}, we include an expository account of the required properties of the free uniform spanning forest and reduce Theorem~\ref{thm:betti} to Theorem~\ref{thm:forests}. The reduction of Theorem~\ref{thm:cost} to Theorem~\ref{thm:forests} in Section~\ref{sec:cost} follows similar lines but using the minimal spanning forest. Strictly speaking, one could reduce Theorem~\ref{thm:betti} to Theorem~\ref{thm:cost} except for the finite generation issue discussed in Section~\ref{sec:further}; we prefer to present a more detailed account of the relation between $L ^2$-Betti numbers and forests and be more concise in Section~\ref{sec:cost}.

Section~\ref{sec:further} discusses the context and further directions of research; we point out for instance that \emph{any} non-amenable finitely generated group admits a random forest with non-trivial (\emph{i.e.}~$>2$) expected degree.

\subsection*{Acknowledgements}
It is a pleasure to thank the following colleagues: Gilles Pisier first mentioned Dixmier's problem to us; Adrian Ioana found a mistake in an earlier draft; Wolfgang L\"uck helped out with a reference. Special thanks to Mikl\'os Ab\'ert for suggesting to use the cost and to Denis Osin for providing the groups mentioned in Corollary~\ref{cor:torsion}.

%===================================================================================================
%===================================================================================================
\section{Forests and Littlewood}\label{sec:forests}
\begin{flushright}
\begin{minipage}[t]{0.5\linewidth}\itshape\small
---?`Usted sin duda querr\'a ver el jard\'in? $[\ldots]$\\
---?`El jard\'in?\\
---El jard\'in de los senderos que se bifurcan.\footnotemark
\end{minipage}
\footnotetext{J.~L.~Borges, \emph{El jard\'in de senderos que se bifurcan}
%; here is the English translation by D.~A.~Yates (\emph{The Garden of Forking Paths}).
%\begin{flushright}
%\vskip-4mm
%\begin{minipage}[t]{0.4\linewidth}\itshape\tiny
%---You no doubt wish to see the garden?$[\ldots]$\\
%---The garden?\\
%---The garden of forking paths.
%\end{minipage}
%\end{flushright}
(\emph{The Garden of Forking Paths}), 1941.%
}
\end{flushright}

\medskip

We follow Serre's conventions~\cite{Serre77} for graphs, which are thus pairs $(V,E)$ of vertex and edge sets with structural maps $E\to V, e\mapsto e_\pm$ and $E\to E, e\mapsto \bar e$. Recall that the underlying ``geometric'' edges consist of pairs of opposed edges $e,\bar e$. In the case of simple graphs, \emph{i.e.} without loops or multiple geometric edges (such as forests), one shall always consider $E$ as a subset of $V\times V$ invariant under the canonical involution and not meeting the diagonal. Recall also that an \emf{orientation} is a fundamental domain for the involution in $E$.

Given a group $G$, we define the space $\graph_G$ of all (simple) graphs on $G$ as the subset $\graph_G\se 2^{G\times G}$ of all subsets $E\se G\times G$ defining a simple graph $(G, E)$. The space $2^{G\times G}$ is compact for the product topology and has a natural $G$-action by left multiplication; since $\graph_G$ is closed and invariant, it is itself a compact $G$-space. A \emf{random graphing} of $G$ is a $G$-invariant probability measure on $\graph_G$.

We now consider the closed $G$-invariant subspace $\F_G\se\graph_G$ of forests and recall from the Introduction that a \emf{random forest} is a random graphing supported on $\F_G$. We shall not be interested in the forest of width zero. We denote by $\F_G^+$ the set of all orientations of all forests and view it as a closed $G$-invariant subspace of the compact $G$-space of subsets of $G\times G$. There is a canonical $G$-equivariant quotient map $\F_G^+\to\F_G$.

\begin{example}
Suppose that $S\se G$ is a subset freely generating a free subgroup. Then we obtain a forest $F\in\F_G$ by $F=\big\{(g, g') : g\inv g' \in S\cup S\inv\}$. This forest is $G$-fixed and hence is a (deterministic) random forest.
\end{example}

\begin{example}\label{ex:UST}
Suppose that $G$ is finite and already endowed with a graph structure $(G,E)$. The uniform measure on the set of all spanning trees of $(G,E)$ is a random forest. Aside from the notion of $G$-invariance, this random forest makes sense for any finite graph $(G,E)$ and is called the \emf{uniform spanning tree}; it will be encountered again in Section~\ref{sec:betti}.
\end{example}

Given a random forest $\mu$ on a group $G$, we denote by $f_\mu(g)$ the probability that $g\in G$ is neighbouring the identity $1\in G$. In other words, $f_\mu(g) = \mu \{F\in\F_G : (1,g)\in F\}$. If $\mu$ has finite width, then $f_\mu$ is a finitely supported function.

We now recall the definition of the $T^1$-norm on the space $\CC[G]$ of finitely supported functions and refer to~\cite{PisierLNM} for details and context. Given $f\in \CC[G]$, on considers all pairs $f^\pm$ of functions $G\times G\to \CC$ such that
$$f(g\inv g') = f^+(g,g') + f^-(g,g') \kern 5mm \forall\, g,g'\in G.$$
The norm $\|f\|_{T^1(G)}$ is the infimum of all such pairs $f^\pm$ of the expression
$$\sup_{g\in G}\sum_{g'\in G} |f^+ (g,g')| + \sup_{g\in G}\sum_{g'\in G} |f^- (g',g)|.$$
The completion of $\CC[G]$ for this norm is a Banach space denoted by $T^1(G)$ that can be realised as functions on $G$. Such functions are called \emf{Littlewood functions} (see \emph{e.g.}~\cite{Varopoulos74, Bozejko87, Bozejko-Fendler91}) in reference to classical harmonic analysis~\cite{Littlewood30}.

\begin{prop}\label{prop:estimates}
Let $\mu$ be a random forest of finite width on a countable group. Then
$$\|f_\mu\|_{T^1(G)} \leq 2 \kern5mm\text{and}\kern5mm  \|f_\mu\|_{\ell^2(G)}\geq \frac{\deg(\mu)}{\sqrt{\width(\mu)}}.$$
\end{prop}

This proposition is a concrete way to carry over to random forests the geometric aspects of a construction for free groups from~\cite{Bozejko-Fendler91}, in accordance with the ideas exposed in the Introduction.

\begin{proof}[Proof of Proposition~\ref{prop:estimates}]
The second inequality is a straightforward application of the Cauchy--Schwarz inequality: setting $S=\{g : f_\mu(g)>0\}$, we have $|S|=\width(\mu)$ and hence
$$\deg(\mu)\ =\ \sum_{g\in G} f_\mu(g)\ =\ \sum_{g\in G} f_\mu(g)\cdot 1_S(g)\ \leq\ \sqrt{\width(\mu)} \,\|f_\mu\|_{\ell^2(G)},$$
as claimed.

\smallskip

We now focus on the first inequality.
Let $\{g_n\}_{n\in \NN}$ be an enumeration of the group $G$. We define a Borel section $\orient:\F_G\to \F_G^+$ as follows. For a forest $F$ and $(g, g')\in F$, let $n$ be the first integer such that $g_n$ belongs to the tree containing $(g, g')$. We then declare that $(g, g')$ belongs to $\orient(F)$ if $g'$ lies between $g$ and $g_n$ in that tree; otherwise, $(g', g)\in \orient(F)$.

We now define two functions $f_\mu^\pm$ on $G\times G$ by
$$f_\mu^+ (g,g') = \mu \big\{ F\in \F_G : (g,g')\in \orient(F) \big\}, \kern5mm f_\mu^-(g,g') = f_\mu^+(g',g).$$
The sum $f_\mu^+ (g,g') + f_\mu^-(g,g')$ is $\mu \big\{ F\in \F_G : (g,g')\in F \big\}$ by the definition of an orientation. Since $\mu$ is $G$-invariant, this quantity depends only on $g\inv g'$ and thus coincides with $f_\mu(g\inv g')$. Therefore, in view of the definition of $T^1(G)$, it remains to justify
$$\sup_{g\in G}\sum_{g'\in G} f_\mu^+ (g,g') \leq 1.$$
Fix thus any $g\in G$. Given a forest $F$, there is at most one $g'\in G$ such that $(g,g')\in\orient(F)$. Indeed, the integer $n$ introduced in the definition of $\orient$ is uniquely determined by $g$ and thus $g'$ can only be the first step towards $g_n$ from $g$, unless $g=g_n$ in which case there is no such $g'$. Therefore $\sum_{g'} f_\mu^+ (g,g')$ is a sum of measures of disjoint subsets of $\F_G$ and hence is bounded by $\mu(\F_G)=1$.
\end{proof}

The space $T^1(G)$ is directly related to uniformly bounded representations:

\begin{prop}\label{prop:inclusion}
If $G$ is unitarisable, then there is a constant $K$ such that
$$\|\cdot\|_{\ell^2(H)}\ \leq\ K\, \|\cdot\|_{T^1(H)}$$
holds for all subgroups $H<G$.
\end{prop}

Observe that the juxtaposition of Propositions~\ref{prop:inclusion} and~\ref{prop:estimates} establishes Theorem~\ref{thm:forests}.

\begin{proof}[Proof of Proposition~\ref{prop:inclusion}]
The fact that unitarisability implies $T^1(G)\se \ell^2(G)$ was established in~\cite[2.3(i)]{Bozejko-Fendler91}, see also Remark~2.8 in~\cite{PisierLNM}. We sketch the main idea for convenience. First, any $T^1$-function gives rise to a uniformly bounded representation on $\ell^2(G)\oplus \ell^2(G)$ by twisting the (diagonal) regular representation with the derivation given by the commutator between the regular representation and kernel operator defined by $f^+$ (using $f^-$ yields the same derivation up to a sign since $f^+ + f^-$ is $G$-invariant). If $G$ is unitarisable, this construction implies that $T^1(G)$ is contained in the space $B(G)$ of matrix coefficients of unitary representations on $G$. Then the stronger conclusion $T^1(G)\se \ell^2(G)$ is obtained by a cotype argument.

Next, we claim that this inclusion is continuous. This follows from the closed graph theorem; indeed, the diagonal in $T^1(G)\times \ell^2(G)$ is closed since it is closed for the weaker topology of pointwise convergence (the latter being Hausdorff).

To conclude the proof, it suffices to show that for all subgroups $H<G$ the canonical inclusion map $\CC[H]\to \CC[G]$ extends to an isometric map $T^1(H)\to T^1(G)$ since the analogous statement for $\ell^2(H)\to \ell^2(G)$ is obvious. Following~\cite[2.7(ii)]{PisierLNM}, we choose a set $R\se G$ of representatives for $G/H$; we arrange that $R$ contains the identity. Given $f\in T^1(H)$, we still write $f:G\to \CC$ for the function extended by zero outside $H$. Let $f^\pm$ be any pair of functions $H\times H\to\CC$ as required by the definition of $T^1(H)$. We now extend the definition of $f^\pm$ to functions $G\times G\to \CC$ by setting
$$f^\pm(g, g')\ =\ 
\begin{cases}
f^\pm(h, h') & \text{if $g=rh, g'=rh'$ for $r\in R$ and $h, h'\in H$,}\\
0 & \text{otherwise.}
\end{cases}$$
The definition is well-posed since $R$ maps injectively to $G/H$. This construction witnesses that $f\in T^1(G)$ with $T^1(G)$-norm bounded by $\|f\|_{T^1(H)}$; the reverse inequality is immediate.
\end{proof}

%===================================================================================================
%===================================================================================================
\section{First \texorpdfstring{$L^2$}{L2}-Betti number}\label{sec:betti}
\begin{flushright}
\begin{minipage}[t]{0.63\linewidth}\itshape\small
To achieve this wonder, electricity is the one and only means. Inestimable good has already been done by the use of this all powerful agent, the nature of which is still a mystery.\footnotemark
\end{minipage}
\footnotetext{N.~Tesla, \emph{The transmission of electrical energy without wires as a means for furthering Peace}, 1905.}
\end{flushright}

\medskip

In 1847, G.~Kirchhoff~\cite{Kirchhoff} proved that given a unit electric current between the endpoints of an edge~$e$ in a finite graph, the current flowing through~$e$ equals the (counting) probability that~$e$ belongs to the uniform spanning tree as introduced in Example~\ref{ex:UST}. There is a well-known connection between currents and combinatorially harmonic functions: see H.~Weyl~\cite{Weyl23} or B.~Eckmann~\cite{Eckmann45VD} and~\cite{Eckmann45} pp.~247--248. This is the starting point for the relation between random forests and the first $L^2$-Betti number that emerged from the work of R.~Pemantle~\cite{Pemantle91}, D.~Gaboriau~\cite{Gaboriau05} and R.~Lyons, exposed in~\cite{Lyons-PeresBOOK}. We shall present just what we need in our setting and refer to~\cite{Lyons-PeresBOOK} and~\cite{Benjamini-Lyons-Peres-Schramm01} for much more material.

\medskip

Let $H$ be a countable group and $S\se H$ some finite subset. Consider the graph $\goth=(H, E)$ obtained by assigning a geometric edge (\emph{i.e.} two opposed elements of $E$) between $h, h'\in H$ whenever $h\inv h'$ is in $S\cup S\inv$. Recall that when $S$ generates $H$, the graph $\goth$ is called a \emf{Cayley graph} for $H$. The left $H$-action preserves the graph structure and we shall investigate random forests arising as subgraphs of $\goth$. Given an enumeration of $H$, let $\goth_n$ be the subgraph of $\goth$ spanned by the first~$n$ elements in $H$. R.~Pemantle~\cite{Pemantle91} showed that the uniform spanning tree measure on $\goth_n$ converges weakly to a measure on $\F_H$. Indeed, it suffices essentially to prove that the probability of the elementary event that a given edge $e$ belongs to a tree in $\goth_n$ (with $n$ large enough to ensure $e\in\goth_n$) is non-increasing in $n$. In view of Kirchhoff's result, this monotonicity follows from Rayleigh's principle stating that added edges can only reduce the current through a given edge. The resulting measure on the space of subgraphs is supported on $\F_G$ since the latter is closed; it is called the \emf{free uniform spanning forest}. (Notice that finite trees can and generally do get disconnected in the limit.) The monotonicity implies in particular that the limit measure does not depend on the enumeration and hence is group-invariant. Much information about this measure can be found in~\cite{Pemantle91, Benjamini-Lyons-Peres-Schramm01, Lyons-PeresBOOK}.

\smallskip

Let $\lalt(\goth)$ be the space of $L^2$-functions on $E$ that change sign under the involution $e\mapsto \bar e$ (\emph{i.e.} ``$1$-forms''). Define the elementary edge function $\chi_e:=\delta_e - \delta_{\bar e}$, where $\delta$ is the Dirac mass. Denote by $d:\ell^2(H)\to\lalt(E)$ the combinatorial derivative (coboundary) defined by $df(e)=f(e_+) - f(e_-)$ and by $d^*$ its adjoint. Let $\lstar(\goth) \se \lalt(\goth)$ be the closure of $d\ell^2(H)$ and $\lcirc(\goth)\se \lalt(\goth)$ the closed span of all cycles (\emph{i.e.} sums $\sum_i\chi_{e_i}$ for sequences $\{e_i\}$ forming cycles). We make the corresponding definitions for the graphs $\goth_n$. The latter being finite, linear algebra provides the orthogonal decomposition $\lalt(\goth_n) = \lstar(\goth_n) \oplus\lcirc(\goth_n)$. The failure of this relation for a general infinite graph $\goth$ is crucial below. Equally important is the fact that whilst $\lalt(\goth_n)$ and $\lcirc(\goth_n)$ clearly densely exhaust $\lalt(\goth)$ and $\lcirc(\goth)$ as $n\to\infty$, the corresponding circumstance does not hold for $\lstar$. (This is the key difference between the present model of \emph{free} random forests and the so-called \emph{wired} case where the finite approximations $\goth_n$ are defined differently.)

\smallskip

We record the following result stated (with all necessary indications for the proof) in the current version of Chapter~10 of the book in progress~\cite{Lyons-PeresBOOK}.

\begin{prop}\label{prop:FUSF}
If $S$ generates $H$, then the expected degree of the free uniform spanning forest is at least $2 \betti(H)$.
\end{prop}

In fact, the exact value $2 \betti(H)+2$ is given in~\cite{Lyons-PeresBOOK}, compare Remark~\ref{rem:ExactValue} below.

\begin{proof}[Proof Proposition~\ref{prop:FUSF}]
We can assume $S=S\inv$ and $1\notin S$ without affecting the statement, so that the neighbours of $1$ in $\goth=(H, E)$ are exactly $S$. Given an edge $e$ and $n$ large enough, denote by $i_n(e)$ the probability that $e$ (or rather the corresponding geometric edge) is in the uniform spanning tree of $\goth_n$. We need to prove
$$\sum_{s\in S} \lim_{n\to\infty} i_n(e_s) \ \geq\ 2 \betti(H), \kern5mm \text{wherein}\kern5mm e_s:=(s, 1).$$

By definition, the first $L^2$-Betti number $\betti(H)$ is the von Neumann dimension of the first $L^2$-cohomology of $H$. The dimension is not affected by passing to the Hausdorff quotient called the \emph{reduced} $L^2$-cohomology. The latter admits a Hodge--de~Rham decomposition which realises the first reduced $L^2$-cohomology of the finitely generated group $H$ as the space
$$\sD_\goth\ :=\ \Big(\lstar(\goth) \oplus\lcirc(\goth)\Big)^\perp\ \se\ \lalt(E)$$
of coboundaries of harmonic functions on vertices~\cite[\S1.1.4]{Lueck}. (In other words $\sD_\goth$ is the space of differentials of harmonic Dirichlet functions, which is isomorphic to the quotient of harmonic Dirichlet functions by the constants.) As for the von Neumann dimension, we recall that for a closed invariant subspace $W<\ell^2(H)$ it is given explicitly by $\pi_W(\delta_1)(1)$, where $\pi_W:\ell^2(H)\to W$ is the orthogonal projection. Combining this with the canonical isometric $H$-identification $\lalt(E) \cong \oplus_{s\in S} \ell^2(H \cdot e_s)$ determined by $\chi_{e_s}\mapsto2\delta_{e_s}$, one has
$$\betti(H)\ =\ \frac12 \sum_{s\in S} \pi_{\sD_\goth}(\chi_{e_s})(e_s),$$
where now $\pi_{\sD_\goth} : \ell^2(E)\to \sD_\goth$. A hurried reader may as well skip the above paragraph and take this identity as \emph{ad hoc} definition of $\betti$.

In view of Kirchhoff's laws, the current on $\goth_n$ yielding unit flow between the endpoints of an edge $e$ is $\pi_{\lstar(\goth_n)}(\chi_e)$ (see \emph{e.g.}~\cite{Eckmann45} p.~248). Therefore, Kirchhoff's characterisation~\cite{Kirchhoff} in terms of the uniform spanning tree shows $i_n(e) = \pi_{\lstar(\goth_n)}(\chi_e)(e)$. Recalling that $\lcirc$, but not $\lstar$, is compatible with the exhaustion, we obtain
\begin{multline*}
\sum_{s\in S} \lim_{n\to\infty} i_n(e_s)\ =\ \sum_{s\in S} \pi_{\lcirc(\goth)^\perp}(\chi_{e_s})(e_s)\ =\ \sum_{s\in S} \pi_{\sD_\goth\oplus\lstar(\goth)}(\chi_{e_s})(e_s)\\
=\ \sum_{s\in S} \pi_{\sD_{\goth}}(\chi_{e_s})(e_s) + \sum_{s\in S} \pi_{\lstar(\goth)}(\chi_{e_s})(e_s).
\end{multline*}
We know already that the first summand equals $2 \betti(H)$. In order to conclude the proof, it remains only to justify that the function $f:=\pi_{\lstar(\goth)}(\chi_{e_s})$ is non-negative at $e_s$. This is the case since (i)~$\chi_{e_s}(e_s)=1$, (ii)~$f$ and $\chi_{e_s}$ are alternating and (iii)~orthogonality imposes $\|f-\chi_{e_s}\|\leq \|\chi_{e_s}\|=2$.
\end{proof}

\begin{remark}\label{rem:ExactValue}
The expected degree is $2 \betti(H)+2$. Indeed, the second summand in the proof above is the expected degree of the \emph{wired} uniform spanning forest on $\goth$ for reasons entirely similar to the above, namely because the exhaustion defining this other model is compatible with~$\lstar$. On the other hand, it is shown in~\cite{Benjamini-Lyons-Peres-Schramm01} that this expected degree is two, using a different characterisation of the wired forest \emph{via} an algorithm of D.~Wilson~\cite{Wilson96}.
\end{remark}

We are now ready to complete the reduction of Theorem~\ref{thm:betti} to Theorem~\ref{thm:forests}. Recall that the \emf{rank} $\rank(H)$ is the minimal number of generators of a group $H$.

\begin{proof}[Proof of Theorem~\ref{thm:betti_bis}, first bound]
Let $H<G$ be a finitely generated subgroup with a generating set $S$ of size $\rank(H)$. The corresponding free uniform spanning forest $\mu$ on $H$ satisfies $\width(\mu)\leq \rank(H)$. Therefore, Proposition~\ref{prop:FUSF} shows that Theorem~\ref{thm:forests} yields the desired bound.
\end{proof}

\begin{proof}[Proof of Theorem~\ref{thm:betti}]
 Let $G$ be any residually finite group with $\betti(G)>0$. Since $G$ is the union of the directed set of all its finitely generated subgroups, Theorem~7.2(3) in~\cite{Lueck} provides us with a finitely generated subgroup $G_0<G$ with $\betti(G_0)>0$. Strictly speaking, one needs to express $G$ as a directed union of \emph{infinite} subgroups in order to apply \emph{loc. cit.}; this is not a restriction since if no such family existed, then $G$ would be amenable as directed union of finite groups, contradicting $\betti(G)>0$ (Theorem~0.2 in~\cite{Cheeger-Gromov}).

The group $G_0$ is still residually finite; we shall use the weaker property that $G_0$ admits finite index subgroups of arbitrarily large index. Notice that for all finite index subgroups $H<G_0$, the quantities $\rank(H)$ and $\betti(H)$ are finite. Moreover, denoting by $[G_0:H]$ the index, one has
$$\betti(H) = [G_0:H]\, \betti(G_0)  \kern5mm\text{and}\kern5mm \rank(H) \leq  [G_0:H]\, \rank(G_0).$$
The above equality is a basic property of $L^2$-Betti numbers~\cite[1.35(9)]{Lueck} whilst the inequality is a (non-optimal) consequence of the Reidemeister--Schreier algorithm, see \emph{e.g.} Proposition~4.1 of~\cite{Lyndon-Schupp} (in fact the quantity $\rank-1$ is sub-multiplicative).

Since $[G_0:H]$ is unbounded, the above (in)equalities violate the first bound of Theorem~\ref{thm:betti_bis}.
\end{proof}

%===================================================================================================
%===================================================================================================
\section{Cost}\label{sec:cost}
This section will be more concise since we shall deduce Theorem~\ref{thm:cost} from our Theorem~\ref{thm:forests} in very much the same way as we did above for Theorem~\ref{thm:betti}.

\medskip

The \emf{cost} $\cost(G)$ of a countable group $G$ is a numerical invariant extensively studied by D.~Gaboriau~\cite{Gaboriau00} (and suggested by G.~Levitt~\cite[p.~1174]{Levitt95}). It is defined as the infimum over all free probability-preserving $G$-actions and over all families of partial isomorphisms generating the resulting orbit equivalence relation of the sum of the measures of the domains of the partial isomorphisms.

We shall not use this definition, but rather the following alternative definition: The cost $\cost(G)$ is the infimum of half the expected degree over all connected random graphings of $G$. The equivalence of the definitions is proved \emph{e.g.} in Proposition~29.5 of~\cite{Kechris-Miller04} (where the definition of the degree differs by a factor~$2$).

\medskip

We now proceed to recall another family of models of random forests, namely the \emf{free minimal spanning forests}; first studied on $\ZZ^d$ in~\cite{Alexander-Molchanov, Alexander95}, it received a general treatment in~\cite{Lyons-Peres-Schramm06}. Let $H$ be a group generated by a finite set $S=S\inv\not\ni 1$ and let $\goth$ be the corresponding Cayley graph $\goth=(H, E)$ (as in Section~\ref{sec:betti}). The free minimal spanning forest associated to this choice $\goth$ is the random graphing of $H$ obtained by assining weights on the (geometric) edges of $\goth$ independently and deleting every edge that has maximal weight in some cycle. We shall need the following fact due to R.~Lyons.

\begin{prop}\label{prop:FMSF}
Let $\mu$ be the above random forest. Then $\deg(\mu)\geq 2 \cost(H)$.
\end{prop}

\begin{proof}
For any $0\leq p \leq 1$, let $\mu_p$ be the random graphing obtained by adding to the $\mu$-random forest each edge of $E$ with probabiliy $p$ independently (thus $\mu_p$ is the union of $\mu$ and of the Bernoulli $p$-percolation random graphing on $\goth$). According to Theorem~1.3 in~\cite{Lyons-Peres-Schramm06}, $\mu_p$ is almost surely connected whenever $p>0$. However, we have by construction
$$\deg(\mu_p)\ \leq\ \deg(\mu) + |S|\cdot p.$$
Letting $p$ tend to zero, the statement follows from the characterisation of $\cost(H)$ recalled above.
\end{proof}

Now the reduction of Theorem~\ref{thm:cost} to Theorem~\ref{thm:forests} proceeds exactly along the lines of the arguments given for Theorem~\ref{thm:betti} in Section~\ref{sec:betti}: First, Proposition~\ref{prop:FMSF} applied to finitely generated subgroups $H$ of $G$ shows that Theorem~\ref{thm:forests} yields the second bound of Theorem~\ref{thm:betti_bis}. Then, one considers finite index subgroups $H<G$ of arbitrarily large index and argues as for Theorem~\ref{thm:betti}, using this time the relation
$$\cost(H)-1\ =\ [G:H]\,(\cost(G) -1),$$
which is Theorem~3 in~\cite{Gaboriau00}. There is no need here to choose a subgroup $G_0$ since $G$ was assumed finitely generated from the outset.

%===================================================================================================
%===================================================================================================
\section{Further considerations}\label{sec:further}
\subsection{}
We begin with a few remarks about the relation between Theorems~\ref{thm:betti} and~\ref{thm:cost}. For any infinite countable group $G$, one has $\cost(G)-1\geq\betti(G)$ (this follows from Corollaire~3.23 in~\cite{GaboriauL2}). A well-known question is whether equality holds. Thus, in the special case of finitely generated groups, Theorem~\ref{thm:cost} is \emph{a priori} stronger than Theorem~\ref{thm:betti}. For general countable groups, one would need the fact that a directed union of cost one groups still has cost one; this is not in the literature (a partial result is Lemme~VI.25 in~\cite{Gaboriau00}), though D.~Gaboriau has orally communicated us a proof. Moreover, the second bound of Theorem~\ref{thm:betti_bis} implies the first.

\smallskip

As for the two types of forests used on finitely generated groups in the reduction of these two theorems to Theorem~\ref{thm:forests}, it is a general fact that on the same Cayley graph, the free \emph{minimal} spanning forest has expected degree bounded below by its \emph{uniform} analogue, see Corollary~1.4 in~\cite{Lyons-Peres-Schramm06}.

\subsection{}\label{sec:WMSF}
As mentioned in the Introduction, \itshape any non-amenable finitely generated group $G$ admits a random forest of expected degree~$>2$\upshape. Indeed, let $S=S\inv\not\ni 1$ a finite generating set. For an integer $k$, consider the $k$th product graph $\goth^{[k]}$ associated to the Cayley graph $\goth$, recalling that it consists of the graph on $G$ where edges correspond to $k$-paths in $\goth$. Strictly speaking, it is a multi-graph, but any forest on $\goth^{[k]}$ can be considered as a forest on the Cayley graph associated to $S^k$. Using spectral isoperimetric estimates, it is proved in~\cite{Pak-Smirnova-Nagnibeda} that the Bernoulli percolation on $\goth^{[k]}$ satisfies $p_c<p_u$ when $k$ is large enough, where $p_c, p_u$ are respectively the critical probability and the uniqueness probability (see~\cite{Lyons-PeresBOOK} for more background). By Proposition~1.7 in~\cite{Lyons-Peres-Schramm06}, this implies that the free minimal spanning forest differs from its \emf{wired} analogue, which implies that the former has higher expected degree by Proposition~3.5 \emph{loc.\ cit}. We recall here that the wired minimal forest is defined exactly as in Section~\ref{sec:cost}, except that one deletes an edge if it has maximal weight even in a cycle ``through infinity'', which is just a bi-infinite path (our reference is still~\cite{Lyons-Peres-Schramm06}). Summing up, it remains only to prove that the expected degree of the wired minimal spanning forest is at least~$2$. In fact, it is exactly~$2$ in the Cayley graph case at hand, see Theorem~3.12 in~\cite{Lyons-Peres-Schramm06}. The above reasoning can be extracted from the arguments of~\cite{Gaboriau-Lyons}.

\subsection{}
It would be desirable to have examples of (residually finite) groups $G$ with $\betti(G)>0$ or $\cost(G)>1$ but not containing $F_2$. We would expect such examples to exist, be it only because the non-vanishing of $\betti$ is a measure-equivalence invariant by~\cite{GaboriauL2}, and $\cost>1$ is so by definition; it seems unlikely that the containment of $F_2$ should be preserved. Interestingly, it is established in~\cite{Peterson-Thom} that for a torsion-free group satisfying a weaker form of the Atiyah conjecture, $\betti>0$ implies the existence of a free subgroup $F_2$. In view of the measure-equivalence invariance of $\betti>0$, one can ask if this statement should be considered as evidence against the Atiyah conjecture. On the other hand, an indication of perhaps surprisingly strong restrictions given by additional algebraic assumptions is M.~Lackenby's result~\cite{Lackenby_detecting} that implies in particular that \emph{residually-$p$-finite finitely presented} groups with $\betti>0$ contain~$F_2$.

\subsection{}
Let $G$ be a group generated by a finite set $S$ and let $\goth$ be the corresponding Cayley graph. Theorem~\ref{thm:forests} is an incentive to find random forests in $\goth$ with large expected degree (compared to the size of $S$). One immediate restriction is given by the \emf{vertex isoperimetric constant} of $\goth$, namely the infimum $i_V(\goth)$ of the ratio $|\partial_V \gothh|/|\gothh|$, where $\gothh$ ranges over all finite subgraphs and $\partial_V$ denotes the vertex-boundary. Indeed, one verifies that the expected degree of any random forest on $\goth$ is bounded by $1+i_V(\goth)/2$. (For the edge-isoperimetric constant, this inequality occurs in~\cite{Lyons-Pichot-Vassout08}.) One can increase at will $i_V$ for any non-amenable graph by replacing $S$ with high powers of that set (as in~\cite{Pak-Smirnova-Nagnibeda}, see~\ref{sec:WMSF} above), but this procedure affects also the denominator in Theorem~\ref{thm:forests}.

\smallskip
Whilst an application of the Hall marriage lemma and of a Cantor--Bernstein argument shows that $\goth$ contains a forest of $n$-regular trees whenever $n\leq i_V(\goth)$, there is no indication that there should be a $G$-invariant measure on the space of such forests.

\subsection{}
Let $G$ be a finitely generated group. Rather than residual finiteness, the proof of Theorem~\ref{thm:betti} (and thus also of Theorem~\ref{thm:cost}) actually uses the existence of infinitely many finite quotients of $G$, or equivalently of some infinite sequence $\{H_n\}$ of finite index subgroups $H_n< G$ which we may assume nested. The Reidemeister--Schreier algorithm quoted earlier shows that the limit
$$\lim_{n\to\infty} \frac{\rank(H_n) - 1}{[G:H_n]}$$
exists; it was introduced in~\cite{Lackenby05} as the \emf{rank gradient}. The \emf{absolute rank gradient} of~\cite{Abert-Nikolov07HEEGARD} is the infimum of the above ratio over all finite index subgroups of $G$.

\smallskip

Does the existence of an infinite sequence with positive rank gradient imply that there are random forests $\mu$ on (subgroups of) $G$ with unbounded ratio $\deg(\mu)^2/\width(\mu)$~? Are there such forests at least when $G$ has positive \emph{absolute} rank gradient?

\smallskip

It follows from the definitions that both $\betti(G)$ and $\cost(G)$ are bounded by $\rank(G)$; therefore, the multiplicativity of $\betti$ and $\cost$ (as recalled in earlier sections) imply that both are lower bounds for the absolute rank gradient. The results of Ab\'ert--Nikolov~\cite{Abert-Nikolov07} suggest some similarity of the rank gradient with the behaviour of these invariants. Moreover, in~\cite{Abert-Nikolov07HEEGARD}, Ab\'ert--Nikolov express the rank gradient of certain chains $\{H_n\}$ as the cost of a specific $G$-action attached to the chain. This result gives added interest to the \emph{fixed price question} which asks whether all relations produced by a given countable group have same cost~\cite{Gaboriau00}.

%===================================================================================================
%\bibliographystyle{amsalpha}
%\bibliography{../BIB/ma_bib}
\def\cprime{$'$}
\providecommand{\bysame}{\leavevmode\hbox to3em{\hrulefill}\thinspace}
\providecommand{\MR}{\relax\ifhmode\unskip\space\fi MR }
\providecommand{\MRhref}[2]{%
  \href{http://www.ams.org/mathscinet-getitem?mr=#1}{#2}
}
\providecommand{\href}[2]{#2}

\end{document}